\documentclass[a4paper,english,fontsize=11pt,parskip=half,abstract=true]{scrartcl}
\usepackage{babel}
\usepackage[utf8]{inputenc}
\usepackage[T1]{fontenc}
\usepackage[a4paper,left=20mm,right=20mm,top=30mm,bottom=30mm]{geometry}
\usepackage{amsmath}
\usepackage{amsthm}
\usepackage{amssymb}
\usepackage{enumerate}
\usepackage{aliascnt}
\usepackage[bookmarks=true,
            pdftitle={The reciprocal character of the conjugation action},
            pdfauthor={Benjamin Sambale},
            pdfkeywords={conjugation action, generalized character},
            pdfstartview={FitH}]{hyperref}

\newtheorem{Thm}{Theorem} 
\newaliascnt{Lem}{Thm}
\newtheorem{Lem}[Lem]{Lemma}
\aliascntresetthe{Lem}
\newaliascnt{Prop}{Thm}
\newtheorem{Prop}[Prop]{Proposition}
\aliascntresetthe{Prop}
\newaliascnt{Cor}{Thm}
\newtheorem{Cor}[Cor]{Corollary}
\aliascntresetthe{Cor}
\newaliascnt{Con}{Thm}

\aliascntresetthe{Con}
\theoremstyle{definition}
\newaliascnt{Def}{Thm}

\aliascntresetthe{Def}
\newaliascnt{Ex}{Thm}

\aliascntresetthe{Ex}

\allowdisplaybreaks[1]

\renewcommand{\phi}{\varphi}
\newcommand{\C}{\mathrm{C}}
\newcommand{\N}{\mathrm{N}}
\newcommand{\Z}{\mathrm{Z}}
\newcommand{\ZZ}{\mathbb{Z}}
\newcommand{\CC}{\mathbb{C}}

\newcommand{\QQ}{\mathbb{Q}}

\newcommand{\FF}{\mathbb{F}}

\newcommand{\pcore}{\mathrm{O}}
\newcommand{\GL}{\operatorname{GL}}
\newcommand{\SL}{\operatorname{SL}}
\newcommand{\SU}{\operatorname{SU}}
\newcommand{\PSU}{\operatorname{PSU}}
\newcommand{\PSp}{\operatorname{PSp}}
\newcommand{\PSL}{\operatorname{PSL}}

\newcommand{\Sz}{\operatorname{Sz}}
\newcommand{\Irr}{\operatorname{Irr}}
\newcommand{\IBr}{\operatorname{IBr}}

\newcommand{\Syl}{\operatorname{Syl}}

\newcommand{\tr}{\operatorname{tr}}

\newcommand{\diag}{\operatorname{diag}}

\mathchardef\ordinarycolon\mathcode`\:  
 \mathcode`\:=\string"8000
 \begingroup \catcode`\:=\active
   \gdef:{\mathrel{\mathop\ordinarycolon}}
 \endgroup

\title{The reciprocal character\\ of the conjugation action}
\author{Benjamin Sambale\footnote{Institut für Algebra, Zahlentheorie und Diskrete Mathematik, Leibniz Universität Hannover, Welfengarten 1, 30167 Hannover, Germany,
\href{mailto:sambale@math.uni-hannover.de}{sambale@math.uni-hannover.de}}}
\date{\today}

\begin{document}
\frenchspacing
\maketitle
\begin{abstract}\noindent
For a finite group $G$ we investigate the smallest positive integer $e(G)$ such that the map sending $g\in G$ to $e(G)|G:\C_G(g)|$ is a generalized character of $G$. It turns out that $e(G)$ is strongly influenced by local data, but behaves irregularly for non-abelian simple groups. We interpret $e(G)$ as an elementary divisor of a certain non-negative integral matrix related to the character table of $G$. Our methods applied to Brauer characters also answers a recent question of Navarro: The $p$-Brauer character table of $G$ determines $|G|_{p'}$.
\end{abstract}


\section{Introduction}

The conjugation action of a finite group $G$ on itself determines a permutation character $\pi$ such that $\pi(g)=|\C_G(g)|$ for $g\in G$. Many authors have studied the decomposition of $\pi$ into irreducible complex characters (see \cite{AdinFrumkin,ChenConj,Fields2,Fields1,Formanek,Frame,HeideTiep,Roichman,Roth,Scharf}). In the present paper we study the reciprocal class function $\tilde\pi$ defined by 
\[\tilde\pi(g):=|\C_G(g)|^{-1}\] 
for $g\in G$.
By a result of Knörr (see \cite[Problem~1.3(c)]{Navarro2} or \autoref{prop} below), 
there exists a positive integer $m$ such that $m\tilde\pi$ is a generalized character of $G$. 
Since $\pi(1)=|G|$, it is obvious that $|G|$ divides $m$. If also $n\tilde\pi$ is a generalized character, then so is $\gcd(m,n)\tilde\pi$ by Euclidean division. 
We investigate the smallest positive integer $e(G)$ such that $e(G)|G|\tilde\pi$ is a generalized character. In most situations it is more convenient to work with the complementary divisor $e'(G):=|G|/e(G)$ which is also an integer by \autoref{prop} below.

We first demonstrate that many local properties of $G$ are encoded in $e(G)$. In the subsequent section we illustrate by examples that most of our theorems cannot be generalized directly. For many simple groups we show that $e'(G)$ is “small”. In the last section we develop a similar theory of Brauer characters. Here we take the opportunity to show that $|G|_{p'}$ is determined by the $p$-Brauer character table of $G$. This answers \cite[Question~A]{NavarroBrauerCT}. Finally, we give a partial answer to \cite[Question~C]{NavarroBrauerCT}.

\section{Ordinary characters}

Our notation follows mostly Navarro's books~\cite{Navarro,Navarro2}. In particular, the set of algebraic integers in $\CC$ is denoted by $\mathbf{R}$. The set of $p$-elements (resp. $p'$-elements) of $G$ is denoted by $G_p$ (resp. $G_{p'}$, deviating from \cite{Navarro}). The usual scalar product of class functions $\chi$, $\psi$ of $G$ is denoted by $[\chi,\psi]=\frac{1}{|G|}\sum_{g\in G}\chi(g)\overline{\psi(g)}$. 
For any real generalized character $\rho$ and any $\chi\in\Irr(G)$ we will often use the fact $[\rho,\chi]=[\rho,\overline{\chi}]$ without further reference.

\begin{Prop}\label{prop}
For every finite group $G$ the following holds:
\begin{enumerate}[(i)]
\item\label{divs} $e(G)$ divides $|G:\Z(G)|$. In particular, $e'(G)$ is an integer divisible by $|\Z(G)|$.
\item\label{two} If $|G|$ is even, so is $e'(G)$. 
\end{enumerate}
\end{Prop}
\begin{proof}\hfill
\begin{enumerate}[(i)]
\item Let $Z:=\Z(G)$. We need to check that $|G||G:Z|[\tilde\pi,\chi]$ is an integer for every $\chi\in\Irr(G)$. Since $\tilde\pi$ is constant on the cosets of $Z$, we obtain
\begin{align*}
|G||G:Z|[\tilde\pi,\chi]&=\sum_{g\in G}\frac{|G:\C_G(g)|\chi(g)}{|Z|}=\sum_{gZ\in G/Z}\frac{|G:\C_G(g)|}{|Z|}\sum_{z\in Z}\chi(gz)\\
&=\sum_{gZ\in G/Z}\frac{|G:\C_G(g)|\chi(g)}{|Z|\chi(1)}\sum_{z\in Z}\chi(z)=[\chi_Z,1_Z]\sum_{gZ\in G/Z}\frac{|G:\C_G(g)|\chi(g)}{\chi(1)}.
\end{align*}
Hence, only the characters $\chi\in\Irr(G/Z)$ can occur as constituents of $\tilde\pi$ and in this case
\[|G||G:Z|[\tilde\pi,\chi]=\sum_{gZ\in G/Z}|G:\C_G(g)|\chi(g)\]
is an algebraic integer. Since the Galois group of the cyclotomic field $\QQ_{|G|}$ permutes the conjugacy classes of $G$ (preserving their lengths), $|G||G:Z|[\tilde\pi,\chi]$ is also rational, so it must be an integer. 

\item Let $|G|$ be even. As in \eqref{divs}, it suffices to show that $|G|^2[\tilde\pi,\chi]$ is even for every $\chi\in\Irr(G)$. Let $\Gamma$ be a set of representatives for the conjugacy classes of $G$.
Let $I$ be a maximal ideal of $\mathbf{R}$ containing $2$. For every integer $m$ we have $m^2\equiv m\pmod{I}$. Hence, 
\begin{align*}
|G|^2[\tilde\pi,\chi]&=\sum_{g\in G}|G:\C_G(g)|\chi(g)=\sum_{x\in \Gamma}|G:\C_G(x)|^2\chi(x)\\
&\equiv \sum_{x\in \Gamma}|G:\C_G(x)|\chi(x)=\sum_{g\in G}\chi(g)=|G|[1_G,\chi]\equiv 0\pmod{I}.
\end{align*}
It follows that $|G|^2[\tilde\pi,\chi]\in\ZZ\cap I=2\ZZ$.\qedhere
\end{enumerate}
\end{proof}

The proof of part \eqref{divs} actually shows that $e(G)|G|\tilde\pi$ is a generalized character of $G/\Z(G)$ and $|G||G:\Z(G)|[\tilde\pi,\chi]$ is divisible by $\chi(1)$.
Part \eqref{two} might suggest that the smallest prime divisor of $|G|$ always divides $e'(G)$. However, there are non-trivial groups $G$ such that $e'(G)=1$. A concrete example of order $3^95^5$ will be constructed in the next section.
We will show later that $e(G)=1$ if and only if $G$ is abelian. 

\goodbreak
\begin{Prop}\label{nil}\hfill
\begin{enumerate}[(i)]
\item\label{direct} For finite groups $G_1$ and $G_2$ we have $e(G_1\times G_2)=e(G_1)e(G_2)$.
\item If $G$ is nilpotent, then $e'(G)=|\Z(G)|$ and every $\chi\in\Irr(G/\Z(G))$ is a constituent of $\tilde\pi$.
\end{enumerate}
\end{Prop}
\begin{proof}\hfill
\begin{enumerate}[(i)]
\item It is clear that $\tilde\pi=\tilde\pi_1\times\tilde\pi_2$ where $\tilde\pi_i$ denotes the respective class function on $G_i$. This shows that $e(G_1\times G_2)$ divides $e(G_1)e(G_2)$. Moreover, $[\tilde\pi,\chi_1\times\chi_2]=[\tilde\pi_1,\chi_1][\tilde\pi_2,\chi_2]$ for $\chi_i\in\Irr(G_i)$. By the definition of $e(G_i)$, the greatest common divisor of $\{e(G_i)|G_i|[\tilde\pi_i,\chi_i]:\chi_i\in\Irr(G_i)\}$ is $1$. In particular, $1$ can be expressed as an integral linear combination of these numbers. Therefore, $1$ is also an integral linear combination of $\{e(G_1)e(G_2)|G_1G_2|[\tilde\pi,\chi_1\times\chi_2]:\chi_i\in\Irr(G_i)\}$. This shows that $e(G_1)e(G_2)$ divides $e(G_1\times G_2)$.

\item By \eqref{direct} we may assume that $G$ is a $p$-group. By \autoref{prop}, $|Z|$ divides $e'(G)$ where $Z:=\Z(G)$. Let $I$ be a maximal ideal of $\mathbf{R}$ containing $p$. Let $\chi\in\Irr(G/Z)$. Since all characters of $G$ lie in the principal $p$-block of $G$, \cite[Theorem~3.2]{Navarro} implies
\[\frac{|G||G:Z|}{\chi(1)}[\tilde\pi,\chi]=\sum_{gZ\in G/Z}\frac{|G:\C_G(g)|\chi(g)}{\chi(1)}\equiv\sum_{gZ\in G/Z}|G:\C_G(g)|\equiv 1\pmod{I}.\]
Therefore, $\chi$ is a constituent of $\tilde\pi$. Taking $\chi=1_G$ yields $|G||G:Z|[\tilde\pi,1_G]\equiv 1\pmod{p}$, so $e'(G)$ is not divisible by $p|Z|$.\qedhere
\end{enumerate}
\end{proof}

We will see in the next section that nilpotent groups cannot be characterized in terms of $e(G)$. 
Moreover, in general not every $\chi\in\Irr(G/\Z(G))$ is a constituent of $\tilde\pi$ (the smallest counterexample is $\mathtt{SmallGroup}(384,5556)$). The corresponding property of $\pi$ was conjectured in \cite{Roth} and disproved in \cite{Formanek}. We do not know any simple group $S$ such that some $\chi\in\Irr(S)$ does not occur in $\tilde\pi$.

Now we study $e(G)$ in the presence of local information. The following reduction to the Sylow normalizer simplifies the construction of examples.

\begin{Lem}\label{red}
Let $P$ be a Sylow $p$-subgroup of $G$ and let $N:=\N_G(P)$. Then $p$ divides $e'(G)$ if and only if $p$ divides $e'(N)$. 
In particular, if $\C_P(N)\ne 1$, then $e'(G)\equiv 0\pmod{p}$. Now suppose that for all $x\in\pcore_{p'}(N)$ we have
\[\sum_{y\in\Z(P)}|H:\C_H(y)|\equiv 0\pmod{p}\]
where $H:=\C_N(x)$. Then $e'(G)\equiv 0\pmod{p}$.
\end{Lem}
\begin{proof}
Let $I$ be a maximal ideal of $\mathbf{R}$ containing $p$. Let $\chi\in\Irr(G)$.
The conjugation action of $P$ on $G$ shows that
\[|G|^2[\tilde\pi,\chi]\equiv\sum_{x\in\C_G(P)}|G:\C_G(x)|\chi(x)\pmod{I}.\]
For $x\in\C_G(P)$, Sylow's Theorem implies
\begin{align*}
|G:\C_G(x)|\equiv|G:\C_G(x)||\C_G(x):\C_N(x)|=|G:N||N:\C_N(x)|\equiv|N:\C_N(x)|\pmod{I}.
\end{align*}
Hence, 
\begin{equation}\label{eq}
|G|^2[\tilde\pi,\chi]\equiv\sum_{x\in\C_G(P)}|N:\C_N(x)|\chi(x)\equiv\sum_{x\in N}|N:\C_N(x)|\chi(x)=|N|^2[\tilde\pi(N),\chi_N]\pmod{I}
\end{equation}
where $\tilde\pi(N)(x):=|\C_N(x)|^{-1}$ for $x\in N$. 
If $e'(N)\equiv 0\pmod{p}$, then the right hand side of \eqref{eq} is $0$ and so is the left hand side. This shows that $e'(G)\equiv 0\pmod{p}$.
If $\C_P(N)\ne 1$, then $e'(N)\equiv 0\pmod{p}$ by \autoref{prop}. 

Now suppose conversely that $e'(G)\equiv 0\pmod{p}$. Since $|G|_p=|N|_p$, it suffices to show that 
\[|G||N|[\tilde\pi(N),\psi]\equiv 0\pmod{I}\] 
for every $\psi\in\Irr(N)$. By an elementary fusion argument of Burnside, elements in $\C_G(P)$ are conjugate in $G$ if and only if they are conjugate in $N$.
Hence, we can define a class function $\gamma$ on $G$ by
\[\gamma(g):=\begin{cases}
\tilde\pi(N)(x)&\text{if $g$ is conjugate in $G$ to }x\in\C_G(P),\\
0&\text{otherwise}
\end{cases}\]
for every $g\in G$.
By \eqref{eq} and Frobenius reciprocity,
\begin{align*}
|G||N|[\tilde\pi(N),\psi]&\equiv|G||N|[\gamma_N,\psi]\equiv|G||N|[\gamma,\psi^G]\equiv\sum_{x\in\C_G(P)}|N:\C_N(x)|\psi^G(x)\\
&\equiv|G|^2[\tilde\pi,\psi^G]\equiv 0\pmod{I}
\end{align*}
as desired.

For the last claim we may assume that $P\unlhd G$ and $N=G$. 
Recall that $\C_G(P)=\Z(P)\times Q$ where $Q=\pcore_{p'}(G)$. Moreover, $\chi(x)\equiv\chi(x_{p'})\pmod{I}$ for every $x\in G$ by \cite[Lemma~4.19]{Navarro2}. Hence,
\[|G|^2[\tilde\pi,\chi]\equiv\sum_{x\in Q}\chi(x)\sum_{y\in\Z(P)}|G:\C_G(xy)|\pmod{I}.\]
Since $\C_G(xy)=\C_G(x)\cap\C_G(y)=\C_H(y)$ where $x\in Q$ and $H:=\C_G(x)$, we conclude that
\[ \sum_{y\in\Z(P)}|G:\C_G(xy)|=|G:H|\sum_{y\in\Z(P)}|H:\C_H(y)|\equiv 0\pmod{I}\]
and the claim follows.
\end{proof}

In the situation of \autoref{red} it is not true that $e'(G)$ and $e'(N)$ have the same $p$-part. In general, $\tilde\pi$ is by no means compatible with restriction to arbitrary subgroups as the reader can convince herself. 

\begin{Lem}\label{knorr}
Let $N:=\pcore_{p'}(G)$. Let $g_p$ be the $p$-part of $g\in G$. Then the map $\gamma:G\to\CC$, $g\mapsto|N:\C_N(g_p)|$ is a generalized character of $G$.
\end{Lem}
\begin{proof}
By Brauer's induction theorem, it suffices to show that the restriction of $\gamma$ to every nilpotent subgroup $H\le G$ is a generalized character of $H$. We write $H=H_p\times H_{p'}$. By a result of Knörr (see \cite[Problem~1.13]{Navarro2}), the restriction $\gamma_{H_p}$ is a generalized character of $H_p$. Hence, also $\gamma_H=\gamma_{H_p}\times 1_{H_{p'}}$ is a generalized character.
\end{proof}

Note that $\Z(G/\pcore_{p'}(G))$ is a $p$-group, since $\pcore_{p'}(G/\pcore_{p'}(G))=1$. In fact, $|\Z(G/\pcore_{p'}(G))|$ is the number of weakly closed elements in a fixed Sylow $p$-subgroup by the $\Z^*$-theorem.
The diagonal monomorphism $G\to\prod_pG/\pcore_{p'}(G)$ embeds $\Z(G)$ into $\prod_p\Z(G/\pcore_{p'}(G))$. Therefore, the following theorem generalizes \autoref{prop}\eqref{divs}.

\begin{Thm}\label{main}
For every prime $p$, $|\Z(G/\pcore_{p'}(G))|$ divides $e'(G)$.
\end{Thm}
\begin{proof}
Let $N:=\pcore_{p'}(G)$, $z:=|\Z(G/N)|$ and $\chi\in\Irr(G)$. Since every element of $G$ can be factorized uniquely into a $p$-part and a $p'$-part, we obtain
\begin{equation}\label{inner}
|G|^2[\tilde\pi,\chi]=\sum_{x\in G_{p'}}\sum_{y\in\C_G(x)_p}|G:\C_G(xy)|\chi(xy).
\end{equation}
We now fix $x\in G_{p'}$ and $H:=\C_G(x)$. In order to show that the inner sum of \eqref{inner} is divisible by $z$ in $\mathbf{R}$ we may assume that $\chi$ is a character of $H$. After decomposing, we may even assume that $\chi\in\Irr(H)$. Since $x\in\Z(H)$, there exists a root of unity $\zeta$ such that $\chi(xy)=\zeta\chi(y)$ for every $y\in H_p$. Moreover, $\C_G(xy)=\C_G(x)\cap\C_G(y)=\C_H(y)$ yields
\[\sum_{y\in H_p}|G:\C_G(xy)|\chi(xy)=\zeta|G:H|\sum_{y\in H_p}|H:\C_H(y)|\chi(y).\]
Let $N_H:=\pcore_{p'}(H)$, $Z^*/N:=\Z(G/N)$, $Z^*_H/N_H:=\Z(H/N_H)$ and $z_H:=|Z^*_H/N_H|$. 
For $x\in Z^*\cap H$ and $h\in H$ we have $[x,h]\in N\cap H\le N_H$. Hence, $Z^*\cap H\le Z^*_H$ and we obtain
\[|Z^*|=|Z^*H:H||Z^*\cap H|\bigm| |G:H||Z^*_H||N:N_H|=|G:H|z_H|N|,\]
i.\,e. $z$ divides $|G:H|z_H$.
Therefore, it suffices to show that 
\begin{equation}\label{claim}
\sum_{y\in H_p}|H:\C_H(y)|\chi(y)\equiv 0\pmod{z_H}
\end{equation}
(the left hand side is an integer since $H_p$ is closed under Galois conjugation).
To this end, we may assume that $H=G$ and $z_H=z$. By \autoref{prop}, there exists a generalized character $\psi$ of $G/N$ such that 
\[\psi(gN)=|G:Z^*||G/N:\C_{G/N}(gN)|\]
for $g\in G$. We identify $\psi$ with its inflation to $G$. 
For $y\in G_p$ it is well-known that $\C_{G/N}(yN)=\C_G(y)N/N$. Let $\gamma$ be the generalized character defined in \autoref{knorr}. Then 
\[(\psi\gamma)(y)=|G:Z^*||G:\C_G(y)N||N:\C_N(y)|=|G:Z^*||G:\C_G(y)|\]
for every $y\in G_p$. 
By a theorem of Frobenius (see \cite[Corollary~7.14]{Navarro2}),
\[\sum_{y\in G_p}|G:Z^*||G:\C_G(y)|\chi(y)=\sum_{y\in G_p}(\psi\tau\chi)(y)\equiv 0\pmod{|G|_p}.\]
It follows that
\[|G:N|_{p'}\sum_{y\in G_p}|G:\C_G(y)|\chi(y)\equiv 0\pmod{z}\]
and \eqref{claim} holds.
\end{proof}

For any set of primes $\sigma$ it is easy to see that $\Z(G/\pcore_{\sigma'}(G))$ embeds into $\prod_{p\in\sigma}\Z(G/\pcore_{p'}(G))$. Hence, \autoref{main} remains true when $p$ is replaced by $\sigma$. 
The following consequence extends \autoref{nil}.

\begin{Cor}\label{pnil}
If $G$ is $p$-nilpotent and $P\in\Syl_p(G)$, then $e'(G)_p=|\Z(P)|$.
\end{Cor}
\begin{proof}
Let $N:=\pcore_{p'}(G)$. Since $G/N\cong P$, \autoref{main} shows that $|\Z(P)|$ divides $e'(G)$. 
For the converse relation, we suppose by way of contradiction that the map 
\[\gamma:G\to\CC,\qquad g\mapsto\frac{1}{p}|G:\Z(P)||G:\C_G(g)|\] 
is a generalized character of $G$. For $x\in P$ we observe that $\C_G(x)=\C_P(x)\C_N(x)$. Hence,
\[(1_P)^G(x)=\frac{1}{|P|}\sum_{\substack{g\in G\\x^g\in P}}1=\frac{1}{|P|}|\C_G(x)||P:\C_P(x)|=|\C_N(x)|.\]
Consequently, $\mu:=(\gamma 1_P^G)_P$ is a generalized character of $P$ such that
\[\mu(x)=\frac{1}{p}|P:\Z(P)||P:\C_P(x)||N|^2\]
for $x\in P$. In the proof of \autoref{nil} we have seen however that
\[[p\mu,1_P]\equiv |N|^2\not\equiv 0\pmod{p}.\]
This contradiction shows that $e'(G)_p$ divides $|\Z(P)|$.
\end{proof}

Next we prove a partial converse of \autoref{pnil}.

\begin{Thm}\label{pnilabel}
For every prime $p$ we have $e(G)_p=1$ if and only if $|G'|_p=1$. In particular, $G$ is abelian if and only if $e(G)=1$.
\end{Thm}
\begin{proof}
If $|G'|_p=1$, then $G/\pcore_{p'}(G)$ is abelian and $e(G)_p=1$ by \autoref{main}. Suppose conversely that $e(G)_p=1$. Then the map $\psi$ with $\psi(g):=|G|_{p'}|G:\C_G(g)|$ for $g\in G$ is a generalized character of $G$.
Let $P$ be a Sylow $p$-subgroup of $G$. Choose representatives $x_1,\ldots,x_k\in P$ for the conjugacy classes of $p$-elements of $G$. Then $\psi(x_i)\equiv\psi(1)\equiv |G|_{p'}\not\equiv 0\pmod{p}$ by \cite[Lemma~4.19]{Navarro2} and $\psi(x_i)^m\equiv 1\pmod{|P|}$ where $m:=\phi(|P|)$ (Euler's totient function). The theorem of Frobenius we have used earlier (see \cite[Corollary~7.14]{Navarro2}) yields
\[k\equiv \sum_{i=1}^k\psi(x_i)^m=|G|_{p'}\sum_{g\in G_p}\psi(g)^{m-1}\equiv 0\pmod{|P|}.\]
In particular, $|P|\le k\le|P|$ and $|P|=k$. It follows that $P$ is abelian and $G$ is $p$-nilpotent by Burnside's transfer theorem. Hence, $G/\pcore_{p'}(G)$ is abelian and $|G'|_p=1$.
\end{proof}

It is clear that $e(G)$ can be computed from the character table of $G$. There is in fact an interesting interpretation:

\begin{Prop}\label{chartable}
Let $X$ be the character table of $G$ and let $Y:=\overline{X}X^\mathrm{t}$. 
Then the following holds:
\begin{enumerate}[(i)]
\item $Y$ is a symmetric, non-negative integral matrix.
\item The eigenvalues of $Y$ are $|\C_G(g)|$ where $g$ represents the distinct conjugacy classes of $G$.
\item $e(G)|G|$ is the largest elementary divisor of $Y$.
\end{enumerate}
\end{Prop}
\begin{proof}
Let $\Irr(G)=\{\chi_1,\ldots,\chi_k\}$. Let $g_1,\ldots,g_k\in G$ be representatives for the conjugacy classes of $G$. 
\begin{enumerate}[(i)]
\item The entry of $Y$ at position $(i,j)$ is
\[\sum_{l=1}^k\overline{\chi_i(g_l)}\chi_j(g_l)=\frac{1}{|G|}\sum_{g\in G}|\C_G(g)|\overline{\chi_i(g)}\chi_j(g)=[\pi,\chi_i\overline{\chi_j}]\ge 0.\]
Now by definition, $Y$ is symmetric.

\item By the second orthogonality relation,
\[\overline{X}^{-1}Y\overline{X}=X^\text{t}\overline{X}=\diag(|\C_G(g_1)|,\ldots,|\C_G(g_k)|).\]

\item It suffices to show that $e(G)|G|$ is the smallest positive integer $m$ such that $mY^{-1}$ is an integral matrix. By the orthogonality relations, $X^{-1}=\bigl(|\C_G(g_i)|^{-1}\overline{\chi_j(g_i)}\bigr)_{i,j=1}^k$. Therefore,
\begin{align*}
Y^{-1}&=(X^{\mathrm{t}})^{-1}\overline{X}^{-1}=\Bigl(\sum_{l=1}^k|\C_G(g_l)|^{-2}\overline{\chi_i(g_l)}\chi_j(g_l)\Bigr)_{i,j}=\Bigl(\frac{1}{|G|}\sum_{l=1}^k|G:\C_G(g_l)|\tilde\pi(g_l)\overline{\chi_i(g_l)}\chi_j(g_l)\Bigr)_{i,j}\\
&=\Bigl(\frac{1}{|G|}\sum_{g\in G}\tilde\pi(g)\overline{\chi_i(g)}\chi_j(g)\Bigr)_{i,j}=\bigl([\tilde\pi,\chi_i\overline{\chi_j},]\bigr)_{i,j}.
\end{align*}
Clearly, $m[\tilde\pi,\chi_i\overline{\chi_j}]$ is an integer for all $i,j$ if and only if $m[\tilde\pi,\chi_i]$ is an integer for $i=1,\ldots,k$. The claim follows.\qedhere
\end{enumerate}
\end{proof}

\section{Examples}

\begin{Prop}
There exist non-trivial groups $G$ such that $e'(G)=1$.
\end{Prop}
\begin{proof}
By \autoref{prop} and \autoref{main} we need a group of odd order such that $\Z(G/\pcore_{p'}(G))=1$ for every prime $p$.
Let $A:=\langle a_1,\ldots,a_4\rangle\cong C_9^4$, $B:=\langle b_1,b_2\rangle\cong C_{25}^2$ and $C:=\langle c\rangle\cong C_{15}$. We define an action of $C$ on $A\times B$ via
\begin{align*}
a_1^c&=a_2^4,&a_2^c&=a_3^4,&a_3^c&=a_4^4,\\
a_4^c&=(a_1a_2a_3a_4)^{-4},&b_1^c&=b_2^6,&b_2^c&=(b_1b_2)^{-6}.
\end{align*}
Note that the action of $c$ on $A$ is the composition of the companion matrix of $X^4+X^3+X^2+X+1$ and the power map $a\mapsto a^4$. In particular, $c^5$ induces an automorphism of order $3$ on $A$. Similarly, $c^3$ induces an automorphism of order $5$ on $B$. 
Now let $G:=(A\times B)\rtimes C$. Then $P:=\langle a_1,\ldots,a_4,c^5\rangle$ is a Sylow $3$-subgroup of $G$ and $Q:=\langle b_1,b_2,c^3\rangle$ is a Sylow $5$-subgroup. It is easy to see that $\C_G(P)=\langle a_1^3,\ldots,a_4^3\rangle$ and $\C_G(Q)=\langle b_1^5,b_2^5\rangle$. 
By the conjugation action of $P$ (resp. $Q$) on $G$, we obtain
\begin{align*}
|G|^2[\tilde\pi,1_G]&=\sum_{g\in G}|G:\C_G(g)|\equiv\sum_{g\in\C_G(P)}|G:\C_G(g)|=1+80\cdot 5\equiv -1\pmod{3}\\
|G|^2[\tilde\pi,1_G]&=\sum_{g\in G}|G:\C_G(g)|\equiv\sum_{g\in\C_G(Q)}|G:\C_G(g)|=1+24\cdot 3\equiv -2\pmod{5}.
\end{align*}
Therefore, $e(G)=|G|$ and $e'(G)=1$.
\end{proof}

Our next example shows that there are non-nilpotent groups $G$ such that $e'(G)=|\Z(G)|$ (take $n=12$ for instance).

\begin{Prop}
Let $G=D_{2n}$ be the dihedral group of order $2n\ge 4$. Then
\[e'(G)=\begin{cases}
4&\text{if }n\equiv 2\pmod{4},\\
2&\text{otherwise}.
\end{cases}\]
\end{Prop}
\begin{proof}
As $G$ is $2$-nilpotent, \autoref{main} shows that $e'(G)_2=4$ if $n\equiv 2\pmod{4}$ and $e'(G)_2=2$ otherwise.
Moreover,
\[|G|^2[\tilde\pi,1_G]=\sum_{g\in G}|G:\C_G(g)|=\begin{cases}
n^2+2n-1&\text{if }2\nmid n,\\
\frac{1}{2}n^2+2n-2&\text{if }2\mid n. 
\end{cases}\]
Since the two numbers on the right hand side have no odd divisor in common with $n$, it follows that $e'(G)_{2'}=1$.  
\end{proof}

For many simple groups it turns out that $e'(G)=2$. 

\begin{Prop}\label{SL2}
For every prime power $q>1$ we have
\begin{align*}
e'(\GL_2(q))&=\begin{cases}
q-1&\text{if }2\nmid q,\\
2(q-1)&\text{if }2\mid q.
\end{cases}\\
e'(\SL_2(q))&=e'(\PSL_2(q))=\begin{cases}
2&\text{if }3\nmid q,\\
6&\text{if }3\mid q.
\end{cases}
\end{align*}
\end{Prop}
\begin{proof}
Suppose first that $G=\GL_2(q)$. By \autoref{prop}, $e'(G)$ is divisible by $|\Z(G)|=q-1$ and by $2(q-1)$ if $q$ is even. The class equation of $G$ is
\[(q^2-1)(q^2-q)=|G|=(q-1)\times 1+\frac{q^2-q}{2}\times (q^2-q)+(q-1)\times(q^2-1)+\frac{(q-1)(q-2)}{2}\times(q^2+q).\]
It follows that
\[
|G||G:\Z(G)|[\tilde\pi,1_G]=1+\frac{(q^2-q)^2}{2}q+(q^2-1)^2+\frac{(q^2+q)^2}{2}(q-2)=q^5-q^3-3q^2+2.
\]
Since 
\begin{equation}\label{coin}
(q^5-q^3-3q^2+2)(1-3q^2)+(q^3-q)(3q^4-q^2-9q)=2,
\end{equation}
we have $\gcd(|G||G:\Z(G)|[\tilde\pi,1_G],|G:\Z(G)|)\le2$ and $e'(G)\le 2(q-1)$. If $q$ is even, we obtain $e'(G)=2(q-1)$ as desired. If $q$ is odd, then $q^5-q^3-3q^2+2$ is odd. Hence, $e'(G)=q-1$ in this case.

Next we assume that $q$ is even and $G=\SL_2(q)=\PSL_2(q)$. The class equation of $G$ is
\[q^3-q=|G|=1\times 1+1\times (q^2-1)+\frac{q}{2}\times q(q-1)+\frac{q-2}{2}\times q(q+1).\]
It follows that
\[
|G|^2[\tilde\pi,1_G]=1+(q^2-1)^2+\frac{q}{2}q^2(q-1)^2+\frac{q-2}{2}q^2(q+1)^2=q^5-q^3-3q^2+2.
\]
By coincidence, \eqref{coin} also shows that $\gcd(|G|^2[\tilde\pi,1_G],|G|)\le2$ and the claim $e'(G)=2$ follows from \autoref{prop}.

Now let $q$ be odd and $G=\SL_2(q)$.
This time the class equation of $G$ is
\[q^3-q=|G|=2\times 1+\frac{q-3}{2}\times q(q+1)+\frac{q-1}{2}\times q(q-1)+4\times\frac{q^2-1}{2}.\]
We obtain
\[
|G|^2[\tilde\pi,1_G]=2+\frac{q-3}{2}q^2(q+1)^2+\frac{q-1}{2}q^2(q-1)^2+(q^2-1)^2=q^5-q^4-q^3-4q^2+3.
\]
Since
\[(q^5-q^4-q^3-4q^2+3)(2-5q^2)+(q^3-q)(5q^4-5q^3-2q^2-23q)=6,\]
it follows that $\gcd(|G|^2[\tilde\pi,1_G],|G|)\in\{2,6\}$. If $3\nmid q$, then 
\[q^5-q^4-q^3-4q^2+3\equiv q-1-q-4+3\equiv1\pmod{3}\] 
and $\gcd(|G|^2[\tilde\pi,1_G],|G|)=2$. In this case, $e'(G)=2$ as desired.

Now let $3\mid q$. Then $e'(G)\mid 6$. It is well-known that the unitriangular matrices form a Sylow $3$-subgroup $P\cong\FF_q$ of $G$. Moreover, $C:=\C_G(P)=P\times\Z(G)\cong P\times\langle -1\rangle$. The normalizer $N:=\N_G(P)$ consists of the upper triangular matrices with determinant $1$. Hence, $\pcore_{3'}(N)=\Z(G)$ and $N/C\cong(\FF_q^\times)^2\cong C_{(q-1)/2}$ acts semiregularly on $P$ via multiplication. It follows that
\[\sum_{y\in P}|N:\C_N(y)|\equiv 1+(q-1)\frac{q-1}{2}\equiv 0\pmod{3}.\]
Thus, \autoref{red} shows $3\mid e'(G)$ and $e'(G)=6$. The final case $G=\PSL_2(q)$ with $q$ odd requires a distinction between $q\equiv\pm1\pmod{4}$, but is otherwise similar. We omit the details.
\end{proof}

\begin{Prop}
For every prime power $q>1$ and $G=\PSU_3(q)$ we have $e'(G)\mid 8$ and $e'(G)=2$ if $q\not\equiv -1\pmod{4}$. 
\end{Prop}
\begin{proof}
The character table of $G$ was computed (with small errors) in \cite{FrameSimpson} based on the results for $\SU(3,q)$. It depends therefore on $\gcd(q+1,3)$. In any event we use GAP~\cite{GAP48} to determine the polynomial $f(q):=|G|^2[\tilde\pi,1_G]$ as in the proof of \autoref{SL2}. It turns out that $\gcd(f(q),|G|)$ always divides $32$. If $q\not\equiv -1\pmod{4}$, then $f(q)$ is not divisible by $4$ and the claim $e'(G)=2$ follows from \autoref{prop}. Now we assume that $q\equiv -1\pmod{4}$. Then $f(q)$ is divisible by $16$ only when $q\equiv 11\pmod{16}$. In this case however, $|G|^2[\tilde\pi,St]$ is not divisible by $16$ where $St$ is the Steinberg character of $G$.
\end{proof}

We conjecture that $e'(\PSU_3(q))=4$ if $q\equiv -1\pmod{4}$. 

\begin{Prop}
For $n\ge 1$ we have $e'(\Sz(2^{2n+1}))=2$.
\end{Prop}
\begin{proof}
Let $q=2^{2n+1}$ and $G=\Sz(q)$. In order to deal with quantities like $\sqrt{q/2}$, we use the generic character table from CHEVIE~\cite{Chevie}. A computation shows that 
\[|G|^2[\tilde\pi,1_G]=q^9-\frac{3}{2}q^8-q^7+\frac{7}{2}q^6-5q^5+\frac{7}{2}q^4-5q^3+\frac{7}{2}q^2-2q+2\equiv 2\pmod{4}\]
and $\gcd(|G|^2[\tilde\pi,1_G],|G|)$ divides $6$. It is well-known that $|G|=q^2(q^2+1)(q-1)$ is not divisible by $3$. Hence, the claim follows from \autoref{prop}.
\end{proof}

For symmetric groups we determine the prime divisors of $e'(S_n)$.

\begin{Prop}\label{sym}
Let $p$ be a prime and let $n=\sum_{i\ge 0} a_ip^i$ be the $p$-adic expansion of $n\ge 1$. Then $p$ divides $e'(S_n)$ if and only if $2a_i\ge p$ for some $i\ge 1$. In particular, $e'(S_n)_p=1$ if $p>2$ and $n<p(p+1)/2$.
\end{Prop}
\begin{proof}
Let $G:=S_n$. For $i\ge 0$ let $P_i$ be a Sylow $p$-subgroup of $S_{p^i}$. Then $P:=\prod_{i\ge 0} P_i^{a_i}$ is a Sylow $p$-subgroup of $G$. By \autoref{red}, it suffices to consider $e'(N)$ where $N:=\N_G(P)$. Since
\[N=\prod_{i\ge 0} \N_{S_{p^i}}(P_i)\wr S_{a_i},\]
we may assume that $n=a_ip^i$ for some $i\ge 1$ by \autoref{nil}. It is well-known that $P_i$ is an iterated wreath product of $i$ copies of $C_p$. It follows that $\Z(P_i)$ has order $p$. Moreover, $\C_G(P)=\Z(P)=\Z(P_i)^{a_i}$. 
For $k=0,\ldots,a_i$ there are exactly $\binom{a_i}{k}(p-1)^k$ elements $(x_1,\ldots,x_{a_i})\in\Z(P)$ such that $|\{i:x_i\ne 1\}|=k$. It is easy to see that these elements form a conjugacy class in $N$. Consequently,
\[\sum_{x\in\Z(P)}|N:\C_N(x)|=\sum_{k=0}^{a_i}\binom{a_i}{k}^2(p-1)^{2k}\equiv\sum_{k=0}^{a_i}\binom{a_i}{k}^2\equiv\binom{2a_i}{a_i}\pmod{p}\]
by the Vandermonde identity. If $2a_i\ge p$, then $\binom{2a_i}{a_i}\equiv 0\pmod{p}$ since $a_i<p$. In this case, \autoref{red} yields $e'(N)\equiv 0\pmod{p}$. Now assume that $2a_i<p$. Then 
\[|N|^2[\tilde\pi(N),1_N]\equiv\sum_{x\in\Z(P)}|N:\C_N(x)|\equiv\binom{2a_i}{a_i}\not\equiv 0\pmod{p}.\]
Hence, $e'(N)_p=1$.
\end{proof}

Based on computer calculations up to $n=45$ we conjecture that
\[e'(S_n)_2=2^{a_1+a_2+\ldots}\]
if $p=2$ in the situation of \autoref{sym}. A(n anonymous) referee noted that this number coincides with $|\Z(P)|$ where $P$ is a Sylow $2$-subgroup of $S_n$. 
We do not know how to describe $e'(S_n)_p$ for odd primes $p$; it seems to depend only on $\lfloor n/p\rfloor$. We also noticed that
\[e'(S_n)=\begin{cases}
e'(A_n)&\text{if }n\equiv 0,1\pmod{4},\\
2e'(A_n)&\text{if }n\equiv 2,3\pmod{4}
\end{cases}\]
for $5\le n\le 45$. This might hold for all $n\ge 5$. In the following tables we list $\tilde e:=e'(G)/2$ for alternating groups and sporadic groups (these results were obtained with GAP).

\[
\begin{array}{*{4}{cc|}cc}
G&\tilde{e}&G&\tilde{e}&G&\tilde{e}&G&\tilde{e}&G&\tilde{e}\\\hline
A_{5}&1&A_{6}&3&A_{7}&3&A_{8}&3&A_{9}&1\\
A_{10}&1&A_{11}&1&A_{12}&2&A_{13}&2&A_{14}&2\\
A_{15}&2\cdot3^{2}\cdot5&A_{16}&3^{2}\cdot5&A_{17}&3^{2}\cdot5&A_{18}&3\cdot5&A_{19}&3\cdot5\\
A_{20}&2\cdot3\cdot5&A_{21}&2\cdot3\cdot5&A_{22}&2\cdot3\cdot5&A_{23}&2\cdot3\cdot5&A_{24}&2\cdot3^{2}\cdot5\\
A_{25}&2\cdot3^{2}&A_{26}&2\cdot3^{2}&A_{27}&2&A_{28}&2^{2}\cdot7&A_{29}&2^{2}\cdot7\\
A_{30}&2^{2}\cdot7&A_{31}&2^{2}\cdot7&A_{32}&7&A_{33}&3\cdot7&A_{34}&3\cdot7\\
A_{35}&3\cdot7&A_{36}&2\cdot7&A_{37}&2\cdot7&A_{38}&2\cdot7&A_{39}&2\cdot 7\\
A_{40}&2\cdot 5\cdot 7&A_{41}&2\cdot 5\cdot 7&A_{42}&2\cdot3^2\cdot5\cdot7&A_{43}&2\cdot3^2\cdot5\cdot7&A_{44}&2^2\cdot 3^2\cdot 5\cdot 7\\
A_{45}&2^2\cdot 3^2\cdot 5\cdot 7
\end{array}
\]

\[
\begin{array}{cc|cc|cc|cc|cc|cc}
G&\tilde{e}&G&\tilde{e}&G&\tilde{e}&G&\tilde{e}&G&\tilde{e}&G&\tilde{e}\\\hline
M_{11}&1&M_{12}&1&J_1&1&M_{22}&1&J_2&5&M_{23}&1\\
HS&1&J_3&1&M_{24}&1&McL&1&He&1&Ru&1\\
Suz&3&ON&1&Co_3&1&Co_2&1&Fi_{22}&1&HN&1\\
Ly&3&Th&1&Fi_{23}&2&Co_1&1&J_4&1&F_{24}'&1\\
B&1&M&1
\end{array}
\]

\section{Brauer characters}

For a given prime $p$, the restriction of our permutation character $\pi$ to the set of $p'$-elements $G_{p'}$ yields a Brauer character $\pi^0$ of $G$. Since $e(G)|G|\tilde\pi$ is a generalized character, there exists a smallest positive integer $f_p(G)$ such that $f_p(G)|G|\tilde\pi^0$ is a generalized Brauer character of $G$. Clearly, $f_p(G)$ divides $e(G)$. As in \cite{Navarro}, we set $[\phi,\mu]^0=\frac{1}{|G|}\sum_{g\in G_{p'}}\phi(g)\overline{\mu(g)}$ for class function $\phi$ and $\mu$ on $G$ (or $G_{p'}$). Recall that for every irreducible Brauer character $\phi\in\IBr(G)$ there exists a projective indecomposable character $\Phi_\phi$ such that $[\Phi_\phi,\mu]^0=\delta_{\phi\mu}$ where $\delta_{\phi\mu}$ is the Kronecker delta (\cite[Theorem~2.13]{Navarro}). 
We first prove the analogue of \autoref{chartable}.

\begin{Prop}\label{brauer}
Let $Y_p:=\overline{X_p}X_p^\mathrm{t}$ where $X_p$ is the $p$-Brauer character table of $G$. Then $Y_p$ is a symmetric, non-negative integral matrix with largest elementary divisor $f_p(G)|G|_{p'}$. In particular, $f_p(G)$ divides $e(G)_{p'}$.
\end{Prop}
\begin{proof}
Let $\IBr(G)=\{\phi_1,\ldots,\phi_l\}$ and $1\le s,t\le l$. Let $g_1,\ldots,g_l$ be representatives for the $p'$-conjugacy classes of $G$. 
Following an idea of Chillag~\cite[Proposition~2.5]{ChillagBr}, we define a non-negative integral matrix $A=(a_{ij})$ by $\phi_i\overline{\phi_s}\phi_t=\sum_{j=1}^la_{ij}\phi_j$. The equation $X_p^{-1}AX_p=\diag(\overline{\phi_s}\phi_t(g_i):i=1,\ldots,l)$ shows that
\[\tr A=\sum_{i=1}^l\overline{\phi_s}(g_i)\phi_t(g_i)=\frac{1}{|G|}\sum_{g\in G_{p'}}\pi(g)\overline{\phi_s}(g)\phi_t(g)=[\pi,\phi_s\overline{\phi_t}]^0\]
is a non-negative integer. At the same time, this is the entry of $Y_p$ at position $(s,t)$. By construction, $Y_p$ is also symmetric.

Now we compute the largest elementary divisor of $Y_p$ by using the projective indecomposable characters $\Phi_i:=\Phi_{\phi_i}$ for $i=1,\ldots,l$. 
For $1\le i,j\le l$ let $a_{ij}:=[\tilde\pi,\Phi_i\overline{\Phi_j}]$. Then
$\sum_{j=1}^la_{ij}\phi_j=(\Phi_i\tilde\pi)^0$ and
\begin{align*}
\sum_{k=1}^la_{ik}[\pi,\phi_k\overline{\phi_j}]^0=\Bigl[\pi,\sum_{k=1}^la_{ik}\phi_k\overline{\phi_j}\Bigr]^0=[\pi,(\Phi_i\tilde\pi)^0\overline{\phi_j}]^0=[\Phi_i,\phi_j]^0=\delta_{ij}.
\end{align*}
Hence, we have shown that $Y_p^{-1}=(a_{ij})$ (notice the similarity to $Y^{-1}$ in the proof of \autoref{chartable}). Since $f_p(G)|G|\tilde\pi^0$ is a generalized Brauer character, it follows that $f_p(G)|G|Y_p^{-1}$ is an integral matrix. In particular, the largest elementary divisor $e$ of $Y_p$ divides $f_p(G)|G|$. 

For the converse relation, recall that $[\phi_i,\phi_j]^0=c_{ij}'$ where $(c_{ij}')$ is the inverse of the Cartan matrix $C$ of $G$. Since $|G|_p$ is the largest elementary divisor of $C$, the numbers $|G|_pc_{ij}'$ are integers. The trivial Brauer character can be expressed as $1_G^0=\sum_{i=1}^lc_{1i}'\Phi_i^0$. Therefore,
\[|G|_pe[\tilde\pi,\Phi_i]=|G|_pe\sum_{j=1}^lc_{1j}'[\tilde\pi\Phi_j,\Phi_i]=\sum_{j=1}^l|G|_pc_{1j}'ea_{ij}\in\ZZ\]
for $i=1,\ldots,l$. Hence, $e|G|_p\tilde\pi^0$ is a generalized Brauer character and $f_p(G)|G|$ divides $e|G|_p$. 
Thus, $f_p(G)|G|_{p'}$ divides $e$. It remains to show that $e$ is a $p'$-number. 

Let $\Irr(G)=\{\chi_1,\ldots,\chi_k\}$ and $X_1:=(\chi_i(g_j))\in\CC^{k\times l}$. Let $Q$ be the decomposition matrix of $G$. Then $X_1=QX_p$ and the second orthogonality relation implies
\[\diag\bigl(|\C_G(g_i)|:i=1,\ldots,l\bigr)=X_1^\text{t}\overline{X_1}=X_p^\text{t}Q^\text{t}Q\overline{X_p}=X_p^\text{t}C\overline{X_p}.\]
By \cite[Corollary~2.18]{Navarro}, we obtain that $\det(Y_p)=|\det(X_p)|^2=(|\C_G(g_1)|\ldots|\C_G(g_l)|)_{p'}$. In particular, $e$ is a $p'$-number.
\end{proof}

In contrast to the ordinary character table, the matrix $X_p^\text{t}\overline{X_p}$ is in general not integral. Even if it is integral, its largest elementary divisor does not necessarily divide $|G|^2$. 
Somewhat surprisingly, $f_p(G)$ can be computed from the ordinary character table as follows.

\begin{Prop}\label{fpord}
The smallest positive integer $m$ such that $|G|_p|G|m[\tilde\pi,\chi]^0\in\ZZ$ for all $\chi\in\Irr(G)$ is $m=f_p(G)$.
\end{Prop}
\begin{proof}
By \cite[Lemma~2.15]{Navarro}, there exists a generalized character $\psi$ of $G$ such that 
\[\psi(g)=\begin{cases}
|G|_p|G|f_p(G)\tilde\pi(g)&\text{if }g\in G_{p'},\\
0&\text{otherwise}.
\end{cases}\]
In particular, $|G|_p|G|f_p(G)[\tilde\pi,\chi]^0=[\psi,\chi]\in\ZZ$ for all $\chi\in\Irr(G)$. Hence, $m$ divides $f_p(G)$. 

Conversely, every $\phi\in\IBr(G)$ can be written in the form $\phi=\sum_{\chi\in\Irr(G)}a_\chi\chi^0$ where $a_\chi\in\ZZ$ for $\chi\in\Irr(G)$ (see \cite[Corollary~2.16]{Navarro}). It follows that $|G|_p|G|m[\tilde\pi,\phi]^0\in\ZZ$ for all $\phi\in\IBr(G)$. This shows that $|G|_p|G|m\tilde\pi^0$ is a generalized Brauer character and $f_p(G)$ divides $|G|_pm$. Since $f_p(G)$ is a $p'$-number, $f_p(G)$ actually divides $m$. 
\end{proof}

In many cases we noticed that $f_p(G)=e(G)_{p'}$. However, the group $G=\PSp_4(5).2$ is an exception with $e(G)_{2'}/f_2(G)=3$. Another exception is $G=\PSU_4(4)$ with $e(G)_{5'}/f_5(G)=3$. 

Now we refine \autoref{pnilabel}.

\begin{Prop}\label{pnilabelbr}
For every prime $q\ne p$ we have $f_p(G)_q=1$ if and only if $|G'|_q=1$.
\end{Prop}
\begin{proof}
If $|G'|_q=1$, then $f_p(G)_q\le e(G)_q=1$ by \autoref{pnilabel}. Suppose conversely, that $f_p(G)_q=1$. Then there exists a generalized Brauer character $\phi$ of $G$ such that $\phi(g)=|G|_{q'}|G:\C_G(g)|$ for $g\in G_{p'}$. As usual there exists a generalized character $\psi$ of $G$ such that $\psi^0=\phi$. Since $G_q\subseteq G_{p'}$ we can repeat the proof of \autoref{pnilabel} at this point.
\end{proof}

Finally, we answer Navarro's question as promised in the introduction. The relevant case ($x=1$) was proved by the author while the extension to $x\in G_{p'}$ was established by G.\,R. Robinson (personal communication). 

\begin{Thm}\label{Cpreg}
The Brauer character table of $G$ determines $|\C_G(x)|_{p'}$ for every $x\in G_{p'}$. 
\end{Thm}
\begin{proof}
It is easy to show that the (Brauer) class function 
\[\rho:=\sum_{\phi\in\IBr(G)}\frac{\Phi_{\phi}(x)}{|\C_G(x)|_p}\overline{\phi}\] 
vanishes off the conjugacy class of $x$ and $\rho(x)=|\C_G(x)|_{p'}$ (see \cite[proof of Theorem~2.13]{Navarro}). 
Thus, it suffices to determine $\rho$ from the Brauer character table $X_p$. By \cite[Lemma~2.21]{Navarro}, $\rho\in \mathbf{R}[\IBr(G)]$. Similarly, by \cite[Lemma~2.15 and Corollary~2.17]{Navarro}, the class function $\theta$, defined to be $|G|_p$ on $G_{p'}$ and $0$ elsewhere, is a generalized projective character of $G$. Moreover, $[\theta,\rho]^0=|G:\C_G(x)|_p$. For every integer $d\ge 2$, we have $\rho(x)/d\notin\ZZ$ or $[\theta,\rho]^0/d\notin\ZZ$. In particular, $\rho/d\notin\mathbf{R}[\IBr(G)]$.

Let $X_p'$ be the matrix obtained from $X_p$ of $G$ by deleting the column corresponding to $x$. Since $X_p$ is invertible, there exists a unique non-trivial solution $v\in\CC^l$ of the linear system $vX_p'=0$ up to scalar multiplication.
We may assume that the components $v_i$ of $v$ are algebraic integers in the cyclotomic field $\QQ_{|G|}$ and that $\sum_{i=1}^lv_i\phi_i(x)$ is a positive rational integer where $\IBr(G)=\{\phi_1,\ldots,\phi_l\}$. We may further assume that $\frac{1}{d}v\notin\mathbf{R}^l$ for every integer $d\ge 2$. Then by the discussion above, we obtain $\rho=\sum_{i=1}^lv_i\phi_i$. 
In particular, 
\[|\C_G(x)|_{p'}=\rho(x)=\sum_{i=1}^l v_i\phi_i(x)\]
is determined by $X_p$.
\end{proof}

G. Navarro made me aware that \autoref{Cpreg} can be used to give a partial answer to \cite[Question~C]{NavarroBrauerCT} as follows.

\begin{Thm}
Let $p\ne q$ be primes such that $q\notin\{3,5\}$. Then the $p$-Brauer character table of a finite group $G$ determines whether $G$ has abelian Sylow $q$-subgroups.
\end{Thm}
\begin{proof}
By \cite{NavarroTiepCon}, $G$ has abelian Sylow $q$-subgroups if and only if $|\C_G(x)|_q=|G|_q$ for every $q$-element $x\in G$. By \cite[Theorem~B]{NavarroBrauerCT}, the columns of the Brauer character table corresponding to $q$-elements can be spotted. Hence, the result follows from \autoref{Cpreg}.
\end{proof}

\section*{Acknowledgment}
I like to thank Christine Bessenrodt and Ruwen Hollenbach for interesting discussions on this topic. I also thank Geoffrey R. Robinson and Gabriel Navarro for sharing their observations with me. Moreover, I appreciate valuable information on CHEVIE by Frank Lübeck. Finally I am grateful to two anonymous referees for pointing out a gap in a proof of a former version of the paper.
The work is supported by the German Research Foundation (\mbox{SA 2864/1-2} and \mbox{SA 2864/3-1}).

\end{document}